\documentclass{amsart}
\usepackage{amsmath,amsfonts,amsthm,amssymb,indentfirst,epic,xurl,graphics,needspace}

\newtheorem{theorem}{Theorem}

\newtheorem{corollary}[theorem]{Corollary}

\newtheorem{definition}[theorem]{Definition}

\numberwithin{equation}{section}
\numberwithin{theorem}{section}

\newcommand{\N}{\mathbb{N}}

\mathchardef\mhyphen="2D

\usepackage[T1]{fontenc} 

\begin{document}

\begin{center}
\texttt{Comments, corrections,
and related references welcomed, as always!}\\[.5em]
{\TeX}ed \today
\\[.5em]
\vspace{2em}
\end{center}

\title%
[Semigroup homomorphisms and projective rank functions]%
{Criteria for existence of semigroup homomorphisms\\
and projective rank functions}
\thanks{%
Readable at \url{http://math.berkeley.edu/~gbergman/papers/}.
After publication, any updates, errata, related references,
etc., found will be noted at
\url{http://math.berkeley.edu/~gbergman/papers/abstracts}\,.\\
\hspace*{1.5em}Most of this work was done in 1990, while the author
was partly supported by NSF contract DMS\ 90-01234.\\
\hspace*{1.5em}Data sharing not applicable, as no datasets were
generated or analyzed in this work.}


\author{George M.\ Bergman}
\address{Department of Mathematics\\
University of California\\
Berkeley, CA 94720-3840, USA}
\email{gbergman@math.berkeley.edu}

\begin{abstract}
Let $P,$ $S,$ and $T$ be semigroups, $f:P\to S$ and
$g:P\to T$ semigroup homomorphisms, and $X$ a generating set for $S$
(possibly infinite).
Clearly, a {\em necessary} condition for there to
exist a homomorphism $S\to T$ making
a commuting triangle with $f$ and $g$ is that for
every relation $f(p) = w(x_1,\,\dots\,,\,x_n)$ holding in $S,$ with
$p\in P,$ $w$ a semigroup word, and $x_1,\,\dots\,,\,x_n \in X,$
there exist $t_1,\,\dots,\,t_n\in T$
satisfying $g(p) = w(t_1,\,\dots\,,\,t_n).$

Under what assumptions will that also be sufficient?
We show that one such family of assumptions is that
(i)~every element of $S$ is a divisor some element of $f(P),$
(ii)~$T$ is right and left cancellative,
(iii)~$T$ is power-cancellative, i.e,
$x^d = y^d \implies x = y$ for $d > 0,$ and
(iv)~a certain technical condition which, in
particular, holds if $T$ admits a semigroup
ordering with the order-type of the natural numbers.

As an application, we obtain an elementary criterion for the
existence of an integer-valued rank function on
finitely generated projective modules over a ring.
\end{abstract}
\maketitle

\section{Main results}\label{S.main}

Here is a bit of notation and terminology that we will use:

\begin{definition}\label{D.divisor}
If $S$ is a semigroup, $S^1$ will denote the
monoid obtained by adjoining an identity element to $S.$

An element $s$ of a semigroup $S$ will be called
a {\em divisor} of an element $t\in S$
if $t = a\,s\,b$ for some $a, b\in S^1.$

An element $s$ of a semigroup $S$ will be called
a {\em weak divisor} of a
subset $A\subseteq S$ if there exists a positive
integer $d$ such that $s^d$ is a divisor of some element
of $A^d$ {\rm(}i.e., of some product $a_1\dots\,a_d$ with $a_i\in A).$
\end{definition}

My only excuse for the peculiar concept of ``weak
divisor'', and for condition~\eqref{d.fin_wk_div}
of the next theorem, which uses it, is that
these are what were needed to abstract an argument I discovered in
considering rank functions on projective modules over a ring.
That application will be made in~\S\ref{S.rank}.
In \S\ref{S.var+eg} we will, inter alia, look at some simpler
conditions that imply~\eqref{d.fin_wk_div}.
In the meantime, I will mention
that~\eqref{d.fin_wk_div} holds frequently, e.g.,
whenever $T$ is a free semigroup or a free abelian semigroup.
Experience may eventually show that one of the stronger conditions
mentioned in~\S\ref{S.var+eg} covers all cases of interest.

Here is our main result.

\begin{theorem}\label{T.main}
Let $P,$ $S$ and $T$ be semigroups, $f:P\to S$ and $g:P\to T$
semigroup homomorphisms, and $X$ a generating set for $S.$
Suppose that
\begin{equation}\begin{minipage}[c]{35pc}\label{d.oAdiv}
every element of $S$ is a divisor of some element of $f(P),$
\end{minipage}\end{equation}
\begin{equation}\begin{minipage}[c]{35pc}\label{d.cancel}
$T$ is right and left cancellative {\rm(}i.e., for $x,y\in T$
and $a,b\in T^1,$ $a\,x\,b\ = a\,y\,b \implies x = y),$
\end{minipage}\end{equation}
\begin{equation}\begin{minipage}[c]{35pc}\label{d.pwr}
$T$ is power cancellative
{\rm(}i.e., $x^d = y^d \implies x = y$ for $x,\,y\in T,$ $d>0),$
\end{minipage}\end{equation}
and
\begin{equation}\begin{minipage}[c]{35pc}\label{d.fin_wk_div}
every finite subset of $g(P) \subseteq T$ has
only finitely many weak divisors in $T$
{\rm(}Definition~\ref{D.divisor} above{\rm)}.
\end{minipage}\end{equation}

Then the following conditions are equivalent:
\begin{equation}\begin{minipage}[c]{35pc}\label{d.realization}
For every relation $f(p) = w(x_1,\,\dots\,,\,x_n)$ holding
in $S,$ with $p\in P,$ $w$ a semigroup word, and
$x_1,\,\dots\,,\,x_n\in X,$ there exist
$t_1,\,\dots\,,\,t_n\in T$ satisfying $g(p) = w(t_1,\,\dots\,,\,t_n).$
\end{minipage}\end{equation}
\begin{equation}\begin{minipage}[c]{35pc}\label{d.oEh}
There exists a homomorphism $h:S\to T$ such that $g=hf.$
\end{minipage}\end{equation}
\end{theorem}

\begin{proof}
Clearly,~\eqref{d.oEh} implies~\eqref{d.realization}
(with $t_i=h(x_i)).$
So what we must prove is the reverse implication.

Observe that a homomorphism $h:S\to T$ is determined by
its restriction to $X;$ and clearly, given any set map
\begin{equation}\begin{minipage}[c]{35pc}\label{d.eta}
$\eta:\,X\to T,$
\end{minipage}\end{equation}
the necessary and sufficient condition for $\eta$ to
extend to a homomorphism $S\to T$ is that
\begin{equation}\begin{minipage}[c]{35pc}\label{d.w12}
for every relation
$w_1(x_1,\,\dots\,,\,x_m) = w_2(x_1,\,\dots,\,x_m)$
holding in $S,$ where $w_1$ and $w_2$ are semigroup words
and $x_1,\,\dots\,,\,x_m\in X,$ the corresponding relation
$w_1(\eta(x_1),\,\dots\,,\,\eta(x_m)) =
w_2(\eta(x_1),\,\dots,\,\eta(x_m))$ holds in~$T.$
\end{minipage}\end{equation}

Given homomorphisms $f:P\to S$ and $g:P\to T$ as in
the statement of the theorem, and a map~\eqref{d.eta},
when will that map determine a homomorphism $h$ such that $g=hf$?
I claim that if
\begin{equation}\begin{minipage}[c]{35pc}\label{d.w=f}
for every relation $f(p)=w(x_1,\,\dots\,,\,x_n)$
holding in $S,$ with $p\in P$ and $x_1,\,\dots,\,x_n\in X,$ one has
$g(p)=w(\eta(x_1),\,\dots,\,\eta(x_n))$ in $T,$
\end{minipage}\end{equation}
then assuming~\eqref{d.oAdiv}-\eqref{d.fin_wk_div},
this will imply that~\eqref{d.w12} holds,
yielding a homomorphism $h:S\to T;$ and clearly,~\eqref{d.w=f}
then shows that $h$ will satisfy the desired relation $g=hf.$

To prove~\eqref{d.w12} assuming~\eqref{d.w=f}, consider
any relation as in the hypothesis of~\eqref{d.w12}.
Let us write $s$ for the common value in $S$
of the two sides of this relation.
Then by~\eqref{d.oAdiv},
$s$ is a divisor of $f(p)$ for some $p\in P.$
Expressing the right and left factors that
carry $s$ to $f(p)$ in terms of
the generating set $X$ (using a list $x_1,\,\dots\,,\,x_n\in X$
extending $x_1,\,\dots\,,\,x_m,$ and allowing empty words
if either or both factors are $1),$ we get a relation
\begin{equation}\begin{minipage}[c]{35pc}\label{d.uw12v}
$\!\!u(x_1,\dots,\,x_n)\,w_1(x_1,\dots,\,x_m)\,
v(x_1,\dots,\,x_n) = f(p) =
u(x_1,\dots,\,x_n)\,w_2(x_1,\dots,\,x_m)\,v(x_1,\dots,\,x_n)\!\!$
\end{minipage}\end{equation}
in $S.$
Regarding this as two equations of the sort appearing
in the hypothesis of~\eqref{d.w=f}, applying~\eqref{d.w=f}
to each, and, since the resulting left-hand sides
are both $g(p),$ equating the right-hand sides,
we get an equation in $T$ which,
by the cancellativity condition~\eqref{d.cancel},
yields the conclusion of~\eqref{d.w12}, as desired.

So,
\begin{equation}\begin{minipage}[c]{35pc}\label{d.iff}
To prove the Theorem, it will suffice to show that
assuming~\eqref{d.realization},
there must exist a map~\eqref{d.eta} for which~\eqref{d.w=f} holds.
\end{minipage}\end{equation}

In constructing such an $\eta,$
the following notation will be useful:
\begin{equation}\begin{minipage}[c]{35pc}\label{d.Sigma}
We shall denote by $\Sigma$ the family of equations
$g(p)=w(\eta(x_1),\,\dots,\,\eta(x_n))$ arising
(as in the conclusion of~\eqref{d.w=f})
from all relations $f(p)=w(x_1,\,\dots\,,\,x_n)$ that
hold in $S,$ with the symbols $\eta(x_1),\,\dots,\,\eta(x_n)$
now regarded as an $\!n\!$-tuple of $\!T\!$-valued unknowns
(though the symbols $g(p)$ will, as in~\eqref{d.w=f},
still denote the indicated constants in~$T).$
\end{minipage}\end{equation}

Below, we shall first deduce from~\eqref{d.realization} that
every {\em finite} subset of $\Sigma$ has a solution
in $T^X,$ then show from this that the whole family $\Sigma$
has such a solution.

To get the result on finite subsets of $\Sigma,$
let us start by showing that given any two equations in $\Sigma,$
\begin{equation}\begin{minipage}[c]{35pc}\label{d.uu'}
$g(p)=u$ \ and \ $g(p')=u',$
\end{minipage}\end{equation}
there exists a nonnegative integer $d$ such that the equation
\begin{equation}\begin{minipage}[c]{35pc}\label{d.u^du'}
$g(p^d\,p')\ =\ u^d\,u',$
\end{minipage}\end{equation}
which also belongs to $\Sigma,$ has the same solution-set
in $T^X$ as the pair of equations~\eqref{d.uu'}.

For this purpose, let $\eta(x_1),\dots,\eta(x_n)$ be the finitely many
variables which appear in one or both of the words
$u$ and $u',$ and consider
assignments of values to $\eta(x_1),\dots,\eta(x_n)$ such
that~\eqref{d.u^du'} holds for at least one value of~$d.$
Every value of $u(\eta(x_1),\dots,\eta(x_n))$
that such an assignment leads to will, in view of~\eqref{d.u^du'},
be a weak divisor (Definition~\ref{D.divisor})
of the $\!2\!$-element set $\{g(p),\ g(p)\,g(p')\},$
so by~\eqref{d.fin_wk_div},
\begin{equation}\begin{minipage}[c]{35pc}\label{d.fin_u-vals}
there are only finitely many values in $T$ that
$u(\eta(x_1),\dots,\eta(x_n))$ can have in solutions
to~\eqref{d.u^du'}, as $d$ ranges over all nonnegative integers.
\end{minipage}\end{equation}

Moreover, I claim that for any such value of
$u(\eta(x_1),\dots,\eta(x_n))$ that is
{\em not} equal to~$g(p),$ the equality~\eqref{d.u^du'} can only
hold for one value of~$d.$
Indeed, assuming that for some $d$ and some
choice of values for $\eta(x_1),\dots,\eta(x_n)$
the equation~\eqref{d.u^du'}
holds, with $g(p)\neq u,$ we see by power-cancellativity
of $T$~\eqref{d.pwr} that %
that choice will make $g(p)^c\neq u^c$ for every $c>0,$ so
if we multiply this inequality on the right by~\eqref{d.u^du'},
right cancellativity of $T$ (condition~\eqref{d.cancel}) shows
that $g(p)^{c+d}g(p')\neq u^{c+d}\,u'.$
So taking for $d$ the least value for which the given
choice of values for $\eta(x_1),\dots,\eta(x_n)$
make~\eqref{d.u^du'} hold, we see that it is the only such value.
Hence, as $u(\eta(x_1),\dots,\eta(x_n))$ ranges over the
finitely many values referred to
in~\eqref{d.fin_u-vals}, only finitely many values of $d$
allow solutions of~\eqref{d.u^du'} with
$g(p)\neq u(\eta(x_1),\dots,\eta(x_n)).$

So let us choose a $d$ that is {\em not} one of those finitely
many values; so that all solutions of~\eqref{d.u^du'}
for that choice of $d$ satisfy $g(p)=u.$
Then by left cancellativity (hypothesis~\eqref{d.cancel}),
those solutions also satisfy $g(p')=u'.$
Thus, for that choice of $d,$ the values of
$\eta(x_1),\dots,\eta(x_n)$ that make~\eqref{d.u^du'} hold are, as
claimed, precisely those that make both equalities of~\eqref{d.uu'}
hold.

Repeatedly applying this result, we see that the solution-set
in $T^X$ of any {\em finite} subset of $\Sigma$ is
equal to the solution-set of a single element of $\Sigma.$
By~\eqref{d.realization}, every member of $\Sigma$
has a nonempty solution-set; hence the above argument
shows that every finite subset of $\Sigma$ does.

Finally, let us show from this that the whole set $\Sigma$
has a nonempty solution-set.

We begin by noting that by~\eqref{d.oAdiv}, each $x\in X$ is
a divisor of some element $f(p_x)$ $(p_x\in P),$ hence
\begin{equation}\begin{minipage}[c]{35pc}\label{d.u_x}
for each $x\in X$
there is some member of $\Sigma,$ say $g(p_x) = u_x,$ such that
the word $u_x$ involves the variable $\eta(x).$
\end{minipage}\end{equation}

For each $x\in X,$ let us fix such an equation $g(p_x)=u_x$ in $\Sigma,$
and denote by $T_x$ the set of divisors in $T$ of the element $g(p_x).$
By~\eqref{d.fin_wk_div} each $T_x$ is finite
(since every divisor of an element $g(p_x)$ is, in particular,
a weak divisor of the singleton set $\{g(p_x)\}).$
Thus, in any solution in the semigroup $T$ of any subsystem
of $\Sigma$ which includes the equation $g(p_x) = u_x,$ the
value given to $\eta(x)$ must be a member of the finite set~$T_x.$

Now let $\Sigma_0$ be any finite subset of $\Sigma,$ let
$\eta(x_1),\dots,\eta(x_n)$ be the variables occurring in
the right-hand sides of the equations in $\Sigma_0,$ and
let $\Sigma_1$ be the finite set obtained by adjoining to
$\Sigma_0$ the additional
equations $g(p_{x_i})=u_{x_i}$ for $i=1,\dots,n.$
Since $\Sigma_1$ is a finite subset
of $\Sigma,$ it will, as we have shown, have
solutions; and since it contains the equations
$g(p_{x_i}) = u_{x_i}$ for $i=1,\dots,n,$
those solutions will have $\!x_i\!$-coordinate in $T_{x_i}$
for $i=1,\dots,n.$
But solutions to $\Sigma_1$ are
also solutions to its subfamily $\Sigma_0,$ so
there are solutions to the latter family
with $\!x_i\!$-coordinate in $T_{x_i}$ for $i=1,\dots,n.$
And since the equations in $\Sigma_0$ involve no variables
but $\eta(x_1),\dots,\eta(x_n),$ we can, in such solutions,
modify the values assigned to all other coordinates
$\eta(x)$ $(x\in X)$
in any way; in particular, replace each such coordinate $\eta(x)$
by a member of the corresponding set $T_x.$
This shows that
\begin{equation}\begin{minipage}[c]{35pc}\label{d.in_prod_T_i}
the solution-set in $T_x$
of every finite subset $\Sigma_0\subseteq\Sigma$
has nonempty intersection with $\prod_{x\in X} T_x.$
\end{minipage}\end{equation}

Let us now regard $\prod_X\,T_x$
as a compact topological space, under the product of the
discrete topologies on the finite factors $T_x,$ and consider each
semigroup word in the variables $\{\eta(x)\,|\,x\in X\}$ as defining a
map $\prod_X\,T_x\to T.$
Regarding $T$ as a discrete space,
each such map is continuous, since it depends on only
finitely many coordinates of $\prod_X\,T_x,$ so the
set of solutions in $\prod_X\,T_x$ to
each equation in $\Sigma$ is closed.

The class of finite subsets of $\Sigma$ is closed
under finite unions, hence the class of their solution-sets in
$\prod_X\,T_x$ is closed under finite intersections.
By the preceding results, all such finite intersections are
nonempty closed sets,
hence the compactness of $\prod_X\,T_x$ shows that the solution-set of
the full family $\Sigma$ is nonempty.
A member of this intersection will be a map~\eqref{d.eta}
satisfying~\eqref{d.w=f},
which by~\eqref{d.iff} completes the proof of the Theorem.
\end{proof}

\section{Application to rank functions on projective modules}\label{S.rank}

The motivation for the above theorem came
from the study of integer-valued
projective rank functions on rings $R,$ i.e.,
functions taking isomorphism classes of finitely
generated projective left $\!R\!$-modules to natural numbers,
which carry direct sums of modules to the sums of the
corresponding natural numbers,
and carry the module $R$ to $1;$ cf.~\cite{DS}.
Let us apply our theorem to that case.

\begin{theorem}\label{T.proj}
Let $R$ be an associative ring, and $X$ a set of
finitely generated nonzero projective left $\!R\!$-modules,
such that every finitely generated projective left $\!R\!$-module is
isomorphic to a direct sum of members of $X.$

Then there exists a nonnegative-integer-valued
projective rank function for $R$
if and only if whenever one has a module isomorphism
\begin{equation}\begin{minipage}[c]{35pc}\label{d.proj}
$R^c\ \cong \ (P_1)^{c_1} \oplus \dots \oplus (P_n)^{c_n}
\ \ \ \ (P_1,\dots,\,P_n\in X,~c,\,c_1,\dots,\,c_n\geq 0),$
\end{minipage}\end{equation}
the integer $c$ is a linear combination
of $c_1,\,\dots\,,\,c_n,$ with nonnegative integer coefficients.

This rank function can be taken to be {\em nondegenerate}
{\rm(}to carry nonzero projective modules to positive
integers{\rm)} if and only if in all such cases,
the integer $c$ can in fact be written as a linear combination
of $c_1,\,\dots\,,\,c_n$ with {\em positive} integer coefficients.
\end{theorem}

\begin{proof}
To get the first assertion, apply the preceding theorem,
taking for both $P$ and $T$ the additive semigroup $\N$ of
nonnegative integers, taking for $S$ the
semigroup of isomorphism classes of nonzero finitely generated
projective $\!R\!$-modules (mumbling the words needed to replace
these proper classes by genuine sets),
for $X$ the set of projective
modules so named above, for $f$ the homomorphism
sending $1$ to the isomorphism class of the
free $\!R\!$-module of rank\ 1, and for $g$ the identity map.
Condition~\eqref{d.oAdiv} of the theorem holds (with ``divisor''
understood to mean ``direct summand'', in view of the
choice of semigroup $S)$ because every finitely generated
projective module is a direct summand in a free module of
finite rank.
Conditions~\eqref{d.cancel} and~\eqref{d.pwr}
(again, translated to additive language) clearly hold
in the semigroup $T=\N,$ and we
see~\eqref{d.fin_wk_div} by noting that an element
of that additive semigroup is a ``divisor'' of another if and only
if it is majorized by that element under the natural ordering
of the integers, and is a ``weak divisor'' of a family if and only
if it is majorized by some member of that family; and
that every nonnegative integer majorizes only finitely many others.
The criterion of Theorem~\ref{T.main} now assumes the desired form
(the ``integer coefficients'' in the statement of the present theorem
corresponding to the $t_i$ of that theorem).

The assertion ot the final paragraph
is obtained by the same argument, using the
semigroup $\N\setminus\{0\}$ rather than $\N$ for $P$ and $T.$
\end{proof}

In~\cite{AS}, generalized projective rank functions, with values
in semigroups $(1/n)\,\N,$ are used to study homomorphisms of
rings $R$ into $n\times n$ matrix rings over division rings.
The same method as above shows that $R$ admits a
$\!(1/n)\,\N\!$-valued projective rank
function if and only if for every isomorphism~\eqref{d.proj}, the
integer $nc$ can be written as a linear
combination of the $c_i$ with nonnegative (or, if we require
the rank function to be nondegenerate, positive) integer coefficients.

\section{Variant conditions and examples}\label{S.var+eg}

In the proof of Theorem~\ref{T.main}, the full strength of
the condition~\eqref{d.fin_wk_div} on weak divisors
was only used in showing that every pair of members of $\Sigma$ has
the same solution-set as a
single member of $\Sigma;$ and that argument was also the only
place where we used the power-cancellativity condition~\eqref{d.pwr}.
In the one later use of condition~\eqref{d.fin_wk_div},
in showing that if every finite subset of $\Sigma$ could
be realized, then so could the whole family $\Sigma,$ that hypothesis
was merely used to show that the set $T_x$ of ordinary divisors
of a single element $g(p_x)$ was finite.
Hence dropping the conditions not used there, we get.

\begin{corollary}[to proof of Theorem~\ref{T.main}]\label{C.varThm}
Assume the hypothesis of Theorem~\ref{T.main}, but without
condition~\eqref{d.pwr} {\rm(}power cancellativity{\rm)}, and
with~\eqref{d.fin_wk_div} weakened to
merely say that every element of $g(P)$ has only
finitely many divisors in $T.$
Then a necessary and sufficient condition
for there to exist a homomorphism $S\to T$ making a commuting
triangle with the given maps
from $P$ is that for every {\em finite set} of relations
\begin{equation}\begin{minipage}[c]{35pc}\label{d.fpi=wi}
$f(p_i)\ =\ w_i(x_1,\,\dots\,,\,x_n)$\hspace{1em}$(i=1,\dots,k)$
\end{minipage}\end{equation}
satisfied in $S$ by elements $x_1,\,\dots\,,\,x_n\in X,$ there
exist $t_1,\,\dots\,,\,t_n\in T$ satisfying the corresponding
relations $g(p_i)=w_i(t_1,\,\dots\,,\,t_n)$ $(i=1,\dots,k).$\qed
\end{corollary}

In the opposite direction, there are conditions stronger
than~\eqref{d.fin_wk_div} but less complicated; so
the statement of Theorem~\ref{T.main} with~\eqref{d.fin_wk_div}
replaced by such a condition remains true, though weaker.
One such condition is gotten by assuming the
conclusion of~\eqref{d.fin_wk_div} for {\em all} finite sets
of elements of $T,$ not just those contained in $g(P):$
\begin{equation}\begin{minipage}[c]{35pc}\label{d.T_wk_div}
Every finite subset of $T$ has only finitely many weak divisors.
\end{minipage}\end{equation}

A much stronger condition, which does not require the
concept of weak divisor, and implies all
of~\eqref{d.cancel},~\eqref{d.pwr} and~\eqref{d.T_wk_div}
(hence~\eqref{d.fin_wk_div}) is
\begin{equation}\begin{minipage}[c]{35pc}\label{d.ord_like_N}
$T$ admits a total ordering having the order-type of the
natural numbers, and satisfying\\
$(\forall\,a, b, c\in T)\ \ a < b \implies
ca < cb$ \ and \ $ac < bc.$
\end{minipage}\end{equation}

That a semigroup $T$ satisfying~\eqref{d.ord_like_N}
satisfies~\eqref{d.cancel},~\eqref{d.pwr} and~\eqref{d.T_wk_div}
will be easy to see once we note
a few elementary properties that~\eqref{d.ord_like_N} implies.
First, any idempotent $e\in T$ must be an identity element.
Indeed, if for any $a$ the product $ea$ were $<a$ or $>a,$
then from~\eqref{d.ord_like_N} we would get the same strict inequality
between $e^2 a$ and $ea,$ contradicting the
idempotence of $e;$ so $e$ is a left identity
element, and by the symmetric argument it is a right identity element.
Second, for all $a,\,b$ with $a$ not
an identity element, we have $ab > b.$
For if we had $b > ab$ or $b = ab,$ we would
get $ab > a^2 b,$ respectively $ab = a^2 b.$
In the former case, we could go on to get an infinite descending
chain $b > ab > a^2 b > a^3 b \dots\,,$
contradicting our order-type hypothesis,
while in the latter we would get $a = a^2,$ so
by our previous observation, $a$ would be an identity element.
Again, we similarly have $ba > b.$
It follows that if $a$ is a divisor of $b,$ then $a\leq b,$
and more generally, that a weak divisor of any set
is one of the finitely many elements of $T$ less than or equal to
the largest element of that set.
Given these properties,~\eqref{d.cancel},~\eqref{d.pwr}
and~\eqref{d.T_wk_div} are easily deduced.

Easy examples of such ordered semigroups are given by the
subsemigroups of the additive group of real numbers generated by
unbounded increasing sequences of positive real numbers, such as
$2\frac{1}{2},\,3\frac{1}{3},\,\dots\,,\,{n+n^{-1},}\,\dots.$
If we take the additive semigroup generated by such a sequence
whose members are also linearly independent, such as the powers of
$\pi,$ that semigroup will be free abelian; thus, free abelian
semigroups on countably many generators satisfy~\eqref{d.ord_like_N}.
The same is true of free (nonabelian) semigroups
on countably many generators;
the interested reader can easily find an appropriate ordering.
But instead of establishing \eqref{d.cancel}-\eqref{d.T_wk_div}
for that class of semigroups in this way, let us, after noting
that~\eqref{d.cancel} and~\eqref{d.pwr} clearly hold for
free semigroups, give another pair of
tools for establishing~\eqref{d.fin_wk_div} and~\eqref{d.T_wk_div}:
\begin{equation}\begin{minipage}[c]{35pc}\label{d.fin_to_1'}
If, for $P,$ $S,$ $T,$ $f,$ $g$ as in the first sentence
of Theorem~\ref{T.main},
$T$ admits a finitely-many-to-one homomorphism $j$
to a semigroup which satisfies~\eqref{d.fin_wk_div} with respect to its
subsemigroup $jg(P),$ then $T$ satisfies~\eqref{d.fin_wk_div} with
respect to its subsemigroup~$g(P).$
\end{minipage}\end{equation}
\begin{equation}\begin{minipage}[c]{35pc}\label{d.fin_to_1}
If a semigroup $T$ admits a finitely-many-to-one homomorphism $j$
to a semigroup which satisfies the condition of~\eqref{d.T_wk_div},
then $T$ also satisfies~\eqref{d.T_wk_div}.
\end{minipage}\end{equation}

These statements are trivial to verify.
Now the free semigroup $T$ on countably
many generators $x_1,\,x_2,\,\dots$ can
be mapped to the free semigroup on one generator $x$ (which
we have seen satisfies~\eqref{d.T_wk_div}) by
sending $x_i$ to $x^i,$ and this map is
finitely-many-to-one; hence by~\eqref{d.fin_to_1},
$T$ satisfies~\eqref{d.T_wk_div}.
The same argument applies to
free semigroups in any semigroup variety $V$ such that
\begin{equation}\begin{minipage}[c]{35pc}\label{d.not_x^m=x^n}
$V$ does not satisfy any identity $x^m = x^n$ with $m\neq n,$
\end{minipage}\end{equation}
(equivalently, varieties in which the free object on
one generator is isomorphic to the semigroup of positive integers.
Some varieties satisfying~\eqref{d.not_x^m=x^n} have, and some
do not have the property that their free semigroups also
satisfy~\eqref{d.cancel} and~\eqref{d.pwr}.
An example which satisfies neither is the variety defined by
the identity saying that all products of two elements
are central: $xyz=zxy.)$

Of course, in any nontrivial variety $V$ of semigroups,
free objects on {\em uncountably} many generators do
not admit finitely-many-to-one homomorphisms to the
free object on one generator.
However, it is not hard to see that a semigroup
satisfies~\eqref{d.T_wk_div}
if and only if all of its countable subsemigroups do.
Hence, {\em all} free semigroups in varieties
satisfying~\eqref{d.not_x^m=x^n} satisfy~\eqref{d.T_wk_div}.

For an example showing that condition~\eqref{d.ord_like_N} on an
ordered semigroup, of having the order-type of the natural numbers,
is stronger than that of being countable and well-ordered,
and also stronger than being generated by a subset
having the order-type of the natural numbers,
let $T$ be the additive subsemigroup of $\mathbb{R}$ generated
by $1/2,\ 2/3,\ 3/4,\,\dots,\,n/(n{+}1),\,\dots\,.$
Since $T$ is generated by a well-ordered
set of positive real numbers, it is
well-ordered~\cite[Theorem~III.2.9, p.\,123]{UA}.
But for any $p\leq q$ in this semigroup, it is not hard to show
that $p$ is a weak divisor of $\{q\}.$
(For the $d$ in the definition of weak divisor,
take a common denominator of $p$ and $q,$ and
remember that since $T$ contains $1/2,$ it
contains all positive integers, so every integer is a
``divisor'' in $T$ of every larger integer.)
Hence, the singleton set $\{1\}$ has infinitely many weak divisors
in this semigroup.
I do not know whether there is a counterexample
to the conclusion of Theorem~\ref{T.main} with this $T$
in place of a $T$ satisfying~\eqref{d.fin_wk_div}.

We remark that in a semigroup $T$
satisfying~\eqref{d.cancel}-\eqref{d.fin_wk_div},
the operation ``the set of weak divisors of'' on
subsets of $T$ need not be idempotent.
For example, within the free semigroup on two generators
$x$ and $y,$ let $T$ be the subsemigroup
generated by $x^2,$ $x^3$ and $y.$
Thus, $T$ consists of those words in $x$ and $y$ in
which every occurrence of $x$ is adjacent to another
occurrence of $x.$
Consider the weak divisors in $T$ of the singleton set $\{x^3 y\}.$
We find that $x^2$ is not a weak divisor of this set:
no power of $x^3 y$ can be factored in such a way that
one of the factors is $x^2$ and the others belong to $T.$
However, $x^3,$ being a divisor of $x^3 y$ in $T,$
is a weak divisor of $\{x^3 y\};$ and
$x^2$ is a weak divisor of any subset of $T$
containing $x^3,$ since $(x^2)^2$ divides $(x^3)^2$ in $T.$
Hence $x^2$ does lie in the set of weak divisors of the
set of weak divisors of $\{x^3 y\}.$

On the other hand, it is not hard to show that on subsets
of a {\em commutative} semigroup, the ``set of weak divisors''
operation {\em is} idempotent; and that in a free semigroup
in the variety of all semigroups,
the weak divisors of a set are simply the divisors
of its members, so that there, too, the
``set of weak divisors'' operation is idempotent.

\end{document}